%% file: main.tex
\documentclass[12pt,a4paper,leqno]{article}

\usepackage[centertags]{amsmath}
\usepackage{amsthm}
\usepackage{amssymb,exscale}
\usepackage{paralist}
\usepackage{latexsym}
\usepackage{graphicx}
\usepackage{hyperref}

\input{preamble}

\begin{document}

\title{Linear response and moderate deviations:\\ hierarchical
  approach. I}

\author{Boris Tsirelson}

\date{}
\maketitle

\begin{abstract}
The Moderate Deviations Principle (MDP) is well-understood for sums of
independent random variables, worse understood for stationary random
sequences, and scantily understood for random fields. Here it is
established for a new class of random processes. The approach is
promising also for random fields.
\end{abstract}

\setcounter{tocdepth}{2}
\tableofcontents

\numberwithin{equation}{section}

\section[Definition, and main result formulated]
  {\raggedright Definition, and main result formulated}
\label{sect1}
\input{sect1}

\numberwithin{equation}{subsection}
\renewcommand{\theequation}{\thesection\alph{subsection}\arabic{equation}}

\section[A chain of H\"older inequalities]
  {\raggedright A chain of H\"older inequalities}
\label{sect2}
\input{sect2}

\section[The chain in action]
  {\raggedright The chain in action}
\label{sect3}
\input{sect3}

\bigskip
\filbreak
{
\small
\begin{sc}
\parindent=0pt\baselineskip=12pt
\parbox{4in}{
Boris Tsirelson\\
School of Mathematics\\
Tel Aviv University\\
Tel Aviv 69978, Israel
\smallskip
\par\quad\href{mailto:tsirel@post.tau.ac.il}{\tt
 mailto:tsirel@post.tau.ac.il}
\par\quad\href{http://www.tau.ac.il/~tsirel/}{\tt
 http://www.tau.ac.il/\textasciitilde tsirel/}
}

\end{sc}
}
\filbreak

\end{document}

%% file: preamble.tex
\numberwithin{equation}{subsection}

\renewcommand{\theequation}{\thesection\alph{subsection}\arabic{equation}}

\swapnumbers
\theoremstyle{definition}
\newtheorem{theorem}[equation]{Theorem}
\newtheorem{lemma}[equation]{Lemma}
\newtheorem{corollary}[equation]{Corollary}
\newtheorem{definition}[equation]{Definition}
\newtheorem{proposition}[equation]{Proposition}
\newtheorem{remark}[equation]{Remark}
\newtheorem{example}[equation]{Example}
\newtheorem{assumption}[equation]{Assumption}

\renewcommand{\phi}{\varphi}

\newcommand{\D}{\mathrm{d}}
\newcommand{\E}{\mathrm{e}}

\renewcommand{\(}{\bigl(}
\renewcommand{\)}{\bigr)\vphantom{)}}

\newcommand{\imply}{\;\;\;\Longrightarrow\;\;\;}

\newcommand{\Median}{\operatorname{Median}}

\newcommand{\const}{\operatorname{const}}

\newcommand{\One}{{1\hskip-2.5pt{\rm l}}}

\newcommand{\eps}{\varepsilon}

\newcommand{\si}{\sigma}
\newcommand{\ga}{\gamma}

\newcommand{\de}{\delta}

\newcommand{\al}{\alpha}
\newcommand{\be}{\beta}
\newcommand{\cO}{\mathcal O}

\newcommand{\la}{\lambda}

\newcommand{\Ex}{\mathbb E\,}
\newcommand{\R}{\mathbb R}

\newcommand{\Z}{\mathbb Z}

\newcommand{\PR}[1]{\mathbb{P}\mskip1.5mu\bigg(\mskip1.5mu#1\mskip1.5mu\bigg)}

\newcommand{\sif}{$\sigma$\nobreakdash-\hspace{0pt}field}

\newcommand{\myatop}[2]{{\genfrac{}{}{0pt}{}{#1}{#2}}}

\newcommand{\close}[1]{$#1$\nobreakdash-\hspace{0pt}close}

%% file: sect1.tex
We examine a class of stationary processes $ X = (X_t)_{t\in\R} $, but
we are interested only in integrals $ \int_\al^\be X_t \, \D t $ rather
than ``individual'' random variables $ X_t $. Continuity of sample
functions is irrelevant as long as these integrals are
well-defined. That is, we merely deal with a two-parameter family of
random variables, denoted (if only for convenience) by
$ \( \int_\al^\be X_t \, \D t \)_{\al<\be} $ and satisfying
\begin{equation}
\int_\al^\be X_t \, \D t + \int_\be^\ga X_t \, \D t = \int_\al^\ga X_t \, \D t
\quad \text{for } -\infty<\al<\be<\ga<\infty \, .
\end{equation}
Stationarity means measure preserving time shifts that send $ \int_\al^\be
X_t \, \D t $ to $ \int_{\al+s}^{\be+s} X_t \, \D t $. Thus, the
distribution of $ \int_\al^\be X_t \, \D t $ depends on $ \be-\al $
only, and we require it to depend measurably:
\begin{equation}\label{01.2}
\text{\small the distribution of } \int_0^r X_t \, \D t \text{ \small
is a measurable function of } r \, ;
\end{equation}
that is, the function $ r \mapsto \Ex \phi \( \int_0^r X_t \, \D t \)
$ is measurable for every bounded continuous $ \phi : \R \to \R $ (or
equivalently, every bounded Borel measurable $\phi$; or just $ \phi
= \One_{(-\infty,s]} $ for all $ s \in \R $; etc). We say that $X$
is \emph{centered,} if
\begin{equation}
\Ex \bigg| \int_\al^\be X_t \, \D t \bigg| < \infty \quad \text{and} \quad \Ex
\int_\al^\be X_t \, \D t = 0 \quad \text{whenever } \al<\be \, .
\end{equation}
We are interested first of all in correlated processes $X$ with continuous
sample paths $ t \mapsto X_t $. However, our general framework admits
uncorrelated processes such as the white noise and the centered
Poisson point process, even though their ``sample paths'' cannot be
interpreted as (usual) functions. For the white noise $X$ the random
variable $ \int_\al^\be X_t \, \D t $ has the normal distribution $
N(0,\be-\al) $. For the centered Poisson point process $X$ the random
variable $ (\be-\al) + \int_\al^\be X_t \, \D t $ has the Poisson
distribution $ P(\be-\al) $.

Our idea of ``not too much correlated'' process is formalized in the
following definition; there, all the four processes ($ X, X^0, X^-,
X^+ $) are interpreted as above. Independence of processes is
independence of the generated \sif s; and the \sif\ generated by $X$
is (by definition) the \sif\ generated by random variables
$ \int_\al^\be X_t \, \D t $. Two processes $X$ and $Y$ are called
identically distributed, if the random vectors
$ \( \int_{\al_1}^{\be_1} X_t \, \D t, \dots, \int_{\al_n}^{\be_n}
X_t \, \D t \) $ and $ \( \int_{\al_1}^{\be_1} Y_t \, \D
t, \dots, \int_{\al_n}^{\be_n} Y_t \, \D t \) $ are identically
distributed whenever $ \al_1<\be_1 $, \dots, $ \al_n < \be_n $.

\begin{definition}\label{definition1}
A centered stationary random process $ X $ satisfying \eqref{01.2}
is \emph{splittable,} if there exist $ r>0 $ and $ \eps > 0 $ such
that $ \Ex \exp \eps | \int_0^r X_t \, \D t | < \infty $,\,\footnote{%
 See also Proposition \ref{2d2} and Remark \ref{2d2a}.}
and there
exists (on some probability space) a triple of random processes $ X^0,
X^-, X^+ $ such that

(a) the two processes $ X^-, X^+ $ are independent;

(b) the four processes $ X, X^0, X^-, X^+ $ are identically distributed;

(c) there exists a number $ c>0 $ such that for all $ a,b>0 $,
\[
\Ex \exp \bigg( c \bigg| \int_{-a}^0 X^-_t \, \D t - \int_{-a}^0 X^0_t \, \D t \bigg|
 + c \bigg| \int_0^b X^+_t \, \D t - \int_0^b X^0_t \, \D t \bigg| \bigg) \le 2 \, .
\]
\end{definition}

\begin{remark}
The class of splittable processes is invariant under rescaling on both
axes ($t$ and $x$), that is, under the transition from $X$ to $Y$
where $ Y_t = a X_{bt} $ for given parameters $ a,b \in (0,\infty) $
(interpreted as $ \int_\al^\be Y_t \, \D t = \frac a
b \int_{b\al}^{b\be} X_t \, \D t $, of course). The same holds for $
a,b \in \R \setminus \{0\} $ (interpreted as $ \int_\al^\be Y_t \, \D
t = -\frac a b \int_{b\be}^{b\al} X_t \, \D t $, if $b<0$).
\end{remark}

\begin{theorem}[\emph{``linear response''}]\label{theorem1}
The following limit exists for every splittable random
process $ X $:
\[
\lim_\myatop{ r\to\infty, \la\to0 }{ \la\log r\to0 }
\frac1{ r\la^2 } \log \Ex \exp \la \int_0^r X_t \, \D t \, .
\]
\end{theorem}

That is, for every $ \eps $ there exist $ R $ and $ \de $ such that the
given expression is \close{\eps} to the limit for all $ r \ge R $
and all $ \la \ne 0 $ such that $ |\la| \log r \le \de $.

We denote this limit by $ \si^2/2 $, $ \si \in [0,\infty) $.

\begin{corollary}[\emph{moderate deviations}]\label{corollary3}
Let $ X $ and $ \si $ be as above, and $ \si \ne 0 $. Then
\[
\lim_\myatop{ r\to\infty, c\to\infty }{ (c\log r)^2/r \to0 }
\frac1{c^2} \log \PR{ \int_0^r X_t \, \D t \ge c\si \sqrt r } = -\frac12 \, .
\]
\end{corollary}

Unfortunately, the region of moderate deviations ($ r \to \infty $, $ c \to
\infty $, $ { \frac{c^2}r \to 0 } $) is not covered. The condition $ \frac{ (c
\log r)^2 }{ r } \to 0 $ leaves a small gap between Corollary \ref{corollary3}
and large deviations ($ \frac{ c^2 } r = \const $).

\begin{corollary}\label{corollary4}
The distribution of $ r^{-1/2} \int_0^r X_t \, \D t $ converges (as
$ r \to \infty $) to the normal distribution $ N(0,\si^2) $.
\end{corollary}

%% file: sect2.tex
\subsection{From a splittable process to cumulant generating functions}

\begin{assumption}\label{2a1}
We restrict ourselves to splittable processes $X$ that satisfy
Def.~\ref{definition1} with $ c=1 $. (This can be ensured, multiplying
a given splittable process by a small positive number).
\end{assumption}

\begin{remark}
Assumption \ref{2a1} is invariant under the transition from $ (X_t)_t
$ to $ (Y_t)_t = (aX_{bt})_t $ provided that $ |a|=|b| $.
\end{remark}

We consider random variables
\begin{equation}
S_r = \frac1{\sqrt r} \int_0^r X_t \, \D t \qquad \text{for } r \in
(0,\infty) \, ,
\end{equation}
and their cumulant generating functions
\begin{equation}\label{2a4}
f_r(\la) = \log \Ex \exp \la S_r \, .
\end{equation}
Note that $ f_r(\la) \ge 0 $, since $ \Ex \exp \la S_r \ge \Ex(1+\la
S_r) = 1 $.

\begin{remark}\label{2a5}
If $ Y_t = aX_{at} $, then $ S_r^{(Y)} = \sqrt a S_{ar}^{(X)} $ and $
f_r^{(Y)}(\la) = f_{ar}^{(X)}(\la \sqrt a) $.
\end{remark}

\begin{example}
(a) If $X$ is the white noise, then $ f_r(\la) = \frac12 \la^2
$. Also, in this case $ (aX_{at})_t $ is distributed like $ (\sqrt a
X_t)_t $.

(b) If $X$ is the centered Poisson point process, then\\
$ f_r(\la) = \(
\E^{\la/\sqrt r} - \frac{\la}{\sqrt r} - 1 \) r $. Note that $
f_r(\la) \to \frac12 \la^2 $ as $ r \to \infty $.
\end{example}

\begin{lemma}\label{1aa1}
For every $ r \in (0,\infty) $ there exist random variables $ U,V,W,Z
$ (on some probability space) such that

$ U,V $ are independent;

$ S_r, U, V $ are identically distributed;

$ S_{2r} $ and $ W $ are identically distributed;

$ \sqrt{2r} W = \sqrt r U + \sqrt r V + Z $;

$ \Ex \exp |Z| \le 2 $.
\end{lemma}

\begin{proof}
We take processes $ X^0, X^-, X^+ $ as in Def.~\ref{definition1} and
let
\[
U = \frac1{\sqrt r} \int_{-r}^0 X_t^- \, \D t \, , \quad
V = \frac1{\sqrt r} \int_{0}^r X_t^+ \, \D t \, , \quad
W = \frac1{\sqrt{2r}} \int_{-r}^r X_t^0 \, \D t
\]
and $ Z = \sqrt{2r} W - \sqrt r U - \sqrt r V $,
then $ |Z| \le \int_{-r}^0 | X_t^- - X_t^0 | \, \D t + \int_0^r |
X_t^+ - X_t^0 | \, \D t $, thus, $ \Ex \exp |Z| \le \Ex \exp \(
\int_{-\infty}^0 | X_t^- - X_t^0 | \, \D t + \int_0^\infty | X_t^+ -
X_t^0 | \, \D t \) \le 2 $.
\end{proof}

Here is a general fact on cumulant generating functions.

\begin{lemma}\label{1aa2}
If a random variable $Z$ satisfies $ \Ex \exp |Z| \le 2 $ and $ \Ex Z
= 0 $, then
\[
\log \Ex \exp \la Z \le \la^2 \qquad \text{for all } \la \in [-1,1] \,
.
\]
\end{lemma}

\begin{proof}
It is sufficient to prove that $ \Ex \( \E^{\la Z} - 1 - \la Z \) \le
\la^2 \( \Ex \E^{|Z|} - 1 \) $; to this end we'll prove that $
\E^{\la z} - 1 - \la z \le \la^2 \( \E^{|z|} - 1 \) $ for all  $ z \in
\R $ and $ \la \in [-1,1] $. WLOG, $ \la \in [0,1] $ (otherwise, use
$(-\la)$ and $(-z)$).

For $ z \ge 0 $ the function $ \la \mapsto \( \E^{\la z} - 1 - \la z
\) / \la^2 = \frac{z^2}{2!} + \frac{z^3}{3!} \la + \dots $ is
increasing on $ (0,1] $, thus, $ \( \E^{\la z} - 1 - \la z \) / \la^2
  \le \E^z - 1 - z $.

For $ z \le 0 $ we have $ \( \E^{\la z} - 1 - \la z \) / \la^2 \le
z^2/2 $, since $ \E^{\la z} - 1 - \la z - \frac12 (\la z)^2 = \frac16
\E^{\theta \la z} (\la z)^3 \le 0 $ for some $ \theta \in [0,1] $.

Finally, for $ z \ge 0 $ we have $ \E^z - 1 - z \le \E^{|z|} - 1 $,
and for  $ z \le 0 $ we have $ z^2/2 \le \E z^2/2 \le \E^{|z|} - 1 $,
since $ \int_0^{|z|} \E t \, \D t \le \int_0^{|z|} \E^t \, \D t $;
indeed, $ \E^t - \E t = \E \( \E^{t-1} - 1 - (t-1) \) \ge 0 $.
\end{proof}

\begin{proposition}\label{2a9}
For all $ r \in (0,\infty) $ and $ p \in (1,\infty) $
\begin{align*}
f_{2r} (\la) &\le \frac2p f_r \Big( \frac{p\la}{\sqrt2} \Big) +
 \frac{p}{p-1} \cdot \frac{\la^2}{2r} \tag{a}
 &&\text{for $ |\la| \le \frac{p-1}p \sqrt{2r} $;} \\
f_{2r} (\la) &\ge 2p f_r \Big( \frac{\la}{p\sqrt2} \Big) -
 \frac{1}{p-1} \cdot \frac{\la^2}{2r} \tag{b}
 &&\text{for $ |\la| \le (p-1) \sqrt{2r} $}.
\end{align*}
\end{proposition}

\begin{proof}
Lemma \ref{1aa1} gives $ U,V,W,Z $. By H\"older's inequality,
\begin{multline*}
\Ex \Big( \exp \frac{\la(U+V)}{\sqrt2} \cdot \exp \frac{\la
  Z}{\sqrt{2r}} \Big) \le \\
\le \Big( \Ex \exp \frac{p\la(U+V)}{\sqrt2} \Big)^{1/p} \Big( \Ex \exp
 \frac{p}{p-1} \frac{\la Z}{\sqrt{2r}} \Big)^{(p-1)/p} \, .
\end{multline*}
We note that
\begin{multline*}
\Ex \exp \frac{p\la(U+V)}{\sqrt2} = \Big( \Ex \exp \frac{p\la
  U}{\sqrt2} \Big) \Big( \Ex \exp \frac{p\la V}{\sqrt2} \Big) = \\
= \Big( \Ex \exp \frac{p\la S_r}{\sqrt2} \Big)^2 = \exp 2f_r \Big(
 \frac{p\la}{\sqrt2} \Big) \, ,
\end{multline*}
\[
\log \Ex \exp \frac{p}{p-1} \frac{\la Z}{\sqrt{2r}} \le \Big(
 \frac{p}{p-1} \Big)^2 \frac{\la^2}{2r} \quad
 \text{for } |\la| \le \frac{p-1}p \sqrt{2r}
\]
(by Lemma \ref{1aa2}), and get (a):
\begin{multline*}
f_{2r}(\la) = \log \Ex \exp \la S_{2r} = \log \Ex \exp \la W = \\
= \log
 \Ex \exp \Big( \frac{\la(U+V)}{\sqrt2} + \frac{\la Z}{\sqrt{2r}} \Big)
\le \frac1p \cdot 2 f_r \Big( \frac{p\la}{\sqrt2} \Big) + \frac{p-1}p
 \Big( \frac{p}{p-1} \Big)^2 \frac{\la^2}{2r} \, .
\end{multline*}
For (b) the argument is similar:
\[
\Ex \bigg( \! \exp \frac{\la W}{p} \cdot \exp \frac{-\la
  Z}{p\sqrt{2r}} \bigg)
\le \Big( \Ex \exp \frac{p\la W}{p}
 \Big)^{1/p} \Big( \Ex \exp \frac{-p\la Z}{(p-1)p\sqrt{2r}}
  \Big)^{(p-1)/p} \, ;
\]
\[
\underbrace{ \log \Ex \exp \frac{\la(U+V)}{p\sqrt2}
 }_{2f_r(\frac{\la}{p\sqrt2})}
\le \frac1p \underbrace{ \log \Ex \exp \la W }_{f_{2r}(\la)} +
\frac{p-1}p \underbrace{ \log \Ex \exp \frac{-\la Z}{(p-1)\sqrt{2r}
  }}_{\le \frac{\la^2}{(p-1)^2 \cdot 2r}} \, . 
\]
\end{proof}

\begin{remark}\label{2a10}
More generally, for all $ r,s \in (0,\infty) $ and $ p \in (1,\infty)
$,
\begin{align*}
f_{r+s} (\la) &\le
\frac1p f_r \bigg( p\la \sqrt{\frac{r}{r+s}} \, \bigg) +
\frac1p f_s \bigg( p\la \sqrt{\frac{s}{r+s}} \, \bigg) +
 \frac{p}{p-1} \cdot \frac{\la^2}{r+s} \tag{a} \\
 &\qquad\qquad\qquad\qquad\text{for $ |\la| \le \frac{p-1}p \sqrt{r+s} $;} \\
f_{r+s} (\la) &\ge
p f_r \bigg( \frac{\la}{p} \sqrt{\frac{r}{r+s}} \, \bigg) +
p f_s \bigg( \frac{\la}{p} \sqrt{\frac{s}{r+s}} \, \bigg) -
\frac{1}{p-1} \cdot \frac{\la^2}{r+s} \tag{b} \\
 &\qquad\qquad\qquad\qquad\text{for $ |\la| \le (p-1) \sqrt{r+s} $}.
\end{align*}
To this end, take $ U = \frac1{\sqrt r} \int_{-r}^0 X_t^- \, \D t $, $
V = \frac1{\sqrt s} \int_{0}^s X_t^+ \, \D t $, $ W =
\frac1{\sqrt{r+s}} \int_{-r}^s X_t^0 \, \D t $ in the proof of
\ref{1aa1}.
\end{remark}

\subsection{Upper bounds}

In this subsection we investigate an arbitrary family of functions $ f_r
: \R \to [0,\infty] $ for $ r \in (0,\infty) $ such that
\begin{equation}\label{start}
f_{2r} (\la) \le \frac2p f_r \Big( \frac{p\la}{\sqrt2} \Big) +
\frac{p}{p-1} \cdot \frac{\la^2}{2r}
\end{equation}
whenever $ 0<r<\infty $, $ 1<p<\infty $ and $ \frac{|\la|}{\sqrt{2r}}
\le \frac{p-1}p $. (The functions \eqref{2a4} satisfy \eqref{start} by
Prop.~\ref{2a9}(a).)

If a family $(f_r)_r$ satisfies \eqref{start}, then for arbitrary $
s \in (0,\infty) $ the rescaled family $ (g_r)_r $ defined by
\begin{equation}\label{1.2}
g_r (\la) = f_{s^2 r} (s\la)
\end{equation}
satisfies \eqref{start} (which is evidently related to Remark \ref{2a5}).

\begin{lemma}\label{1a1}
Let $ a \ge 1 $, $ \eps \ge 0 $, $ r>0 $, and $ \frac{\eps}{\sqrt r} \le \sqrt2 -
1 $. If
\[
f_r(\eps\la) \le (a-1) \la^2 \quad \text{for } |\la| \le 1 \, ,
\]
then
\[
f_{2r} (\eps\la) \le \bigg( a \Big( 1 + \frac{\eps}{\sqrt r} \Big)
-1 \bigg) \la^2 \quad \text{for } |\la| \le 1 \, .
\]
\end{lemma}

\begin{remark}\label{invar}
If this lemma holds for $\eps$ and $r$, then for arbitrary $ s \in
(0,\infty) $ it holds also for $ s\eps $ and $ s^2 r $ due to the
rescaling \eqref{1.2}. All relevant functions of $ \eps, r $ depend
only on the invariant combination $ \eps/\sqrt r $. (Also $a$ and
$\la$ are invariant.) Therefore it is sufficient to prove Lemma
\ref{1a1} for $r=1$ only. (This argument will be used many times.)
\end{remark}

\begin{proof}[Proof of Lemma \ref{1a1}]
We restrict ourselves to the case $r=1$ according to Remark
\ref{invar}. Assuming $ \eps \ne 0 $ we take $ p = 1 + \eps $, note
that $ p \le \sqrt2 $, $ \frac{p-1}p \ge \frac{\eps}{\sqrt2} $, and
apply \eqref{start} to $ \eps\la $ in place of $\la$, getting two
summands. The second summand is $ \frac{p}{p-1} \frac{\eps^2 \la^2}{2}
\le \frac{ \eps \la^2 }{ \sqrt2 } \le \eps \la^2 $. The first summand
does not exceed $ \frac2{1+\eps} (a-1) \frac12 (1+\eps)^2 \la^2 \le
(1+\eps) (a-1) \la^2 $.
\end{proof}

Iterating the transition $ r \mapsto 2r $ we multiply $a$ by $ \( 1
+ \frac{\eps}{\sqrt r} \) \( 1 + \frac{\eps}{\sqrt{2r}} \) \( 1
+ \frac{\eps}{\sqrt{4r}} \) \dots \le \exp \( \frac{\sqrt2}{\sqrt2-1}
\frac{\eps}{\sqrt r} \) $ and get the following.

\begin{proposition}\label{1a2}
Let $ a \ge 1 $, $ \eps \ge 0 $, $ r>0 $, and $ \frac{\eps}{\sqrt
r} \le \sqrt2 - 1 $. If
\[
f_r(\eps\la) \le (a-1) \la^2 \quad \text{for } |\la| \le 1 \, ,
\]
then, for every $ n=0,1,2,\dots $,
\[
f_{2^n r} (\eps\la) \le \bigg( a \exp \Big( \frac{\sqrt2}{\sqrt2-1}
\frac{\eps}{\sqrt r} \Big) -1 \bigg) \la^2 \quad \text{for } |\la| \le
1 \, .
\]
\end{proposition}

\begin{lemma}\label{1a3}
Let $ a,b,c,\de \ge 0 $, $ b\de < 1 $, and $ r>0 $. If
\[
f_r(\la) \le \frac{ a\la^2 }{ 1 - \frac{b|\la|}{\sqrt r} } +
\frac{c|\la|}{\sqrt r} \quad \text{for } |\la| \le \de \sqrt r \, ,
\]
then
\[
f_{2r}(\la) \le \frac{ a\la^2 }{ 1 - \frac{(b+1)|\la|}{\sqrt{2r}} } +
\frac{(2c+1)|\la|}{\sqrt{2r}} \quad \text{for } |\la| \le
\frac{\de}{1+\de} \sqrt{2r} \, .
\]
\end{lemma}

\begin{proof} \let\qed\relax
We restrict ourselves to the case $r=1$ according to Remark
\ref{invar}.\footnote{%
 Invariant are $b$, $c$, $\de$, $ \la^2/r $, $ a\la^2 $.}
Assuming $ \la\ne0 $ we take
\[
p = \frac{1}{ 1 - \frac{|\la|}{\sqrt{2}} } \, ,
\]
note that
\begin{compactitem}
\item  $ 1 - \frac{bp|\la|}{\sqrt{2}} = p \( 1 -
\frac{(b+1)|\la|}{\sqrt{2}} \) $ (since $ 1 = p - p
\frac{|\la|}{\sqrt{2}} $);
\item $ \big| \frac{p\la}{\sqrt2} \big| = \frac{
  \frac{|\la|}{\sqrt{2}} }{ 1 - \frac{|\la|}{\sqrt{2}} } \le \de $;
\item $ \frac{p-1}p = \frac{|\la|}{\sqrt{2}} $;
\end{compactitem}
and apply \eqref{start}, getting two summands. The second summand is
$ \frac{p}{p-1} \frac{\la^2}{2} = \frac{|\la|}{\sqrt{2}} $. The
first summand is
\begin{multline*}
\frac2p f_1 \Big( \frac{p\la}{\sqrt2} \Big) \le \frac2p \bigg( \frac{
 a\(\frac{p\la}{\sqrt2}\)^2 }{ 1 - b \big|\frac{p\la}{\sqrt2}\big|
 } + c \Big|\frac{p\la}{\sqrt2}\Big| \bigg) = \\
= \frac{ap\la^2}{1 - \frac{bp|\la|}{\sqrt{2}}} +
 \frac{2c|\la|}{\sqrt{2}} = \frac{ a\la^2 }{ 1 -
 \frac{(b+1)|\la|}{\sqrt{2}} } + \frac{2c|\la|}{\sqrt{2}} \,
 . \qquad \rlap{$\qedsymbol$}
\end{multline*}
\end{proof}

\begin{proposition}\label{1a4}
Let $ a,\de \ge 0 $, and $ r>0 $. If
\[
f_r(\la) \le a\la^2 \quad \text{for } |\la| \le \de \sqrt r \, ,
\]
then (for every $ n=0,1,2,\dots $)
\[
f_{2^n r}(\la) \le \frac{ a\la^2 }{ 1 - \frac{n|\la|}{2^{n/2}\sqrt r} } +
\frac{2^{n/2}|\la|}{\sqrt r} \quad \text{for }
|\la| \le \frac{\de}{1+n\de} 2^{n/2} \sqrt r \, .
\]
\end{proposition}

\begin{proof} \let\qed\relax
We prove a bit stronger inequality, with the second summand $
(1-2^{-n}) \frac{2^{n/2}|\la|}{\sqrt r} $ instead of
$ \frac{2^{n/2}|\la|}{\sqrt r} $, by induction in $n$. Case $n=0$ is
trivial. If the claim holds for $n$, then Lemma \ref{1a3} applies to $
2^n r $, $b=n$, $c=(1-2^{-n}) 2^n = 2^n-1 $, and $ \frac{\de}{1+n\de}
$, giving
\[
f_{2^{n+1} r}(\la) \le
\frac{ a\la^2 }{ 1 - \frac{(n+1)|\la|}{\sqrt{2^{n+1} r} } } +
\frac{(2^{n+1}-1)|\la|}{\sqrt{2^{n+1} r} } \quad \text{for }
 |\la| \le \frac{\de}{1+(n+1)\de} \sqrt{2^{n+1} r} \, . \qquad
 \rlap{$\qedsymbol$}
\]
\end{proof}

\begin{theorem}\label{1a5}
Let $ \eps \in (0,\sqrt2-1] $ and $ r,a \in (0,\infty) $. If
\[
f_r(\eps\la) \le a \la^2 \quad \text{for } |\la| \le \sqrt r \, ,
\]
then, for every $ n = 1,2,\dots $,
\[
f_{2^n r} (\eps\la) \le a \la^2 + C \eps \Big( a
+ \frac1r \Big) \frac{1+V}{1-\eps V} \la^2 \qquad \text{for } \;
|\la| \le \frac{ 2^{n/2} \sqrt r }{ \eps n + \max(\eps\sqrt{2n},1)
} \, ,
\]
where
\[
C = \frac1\eps \bigg( \exp \Big( \frac{ \sqrt2 }{ \sqrt2 - 1
} \eps \Big) - 1 \bigg) \, , \quad V = \frac{ n }{ 2^{n/2} } \frac{
|\la| }{ \sqrt r } \, .
\]
\end{theorem}

\medskip

Note that the condition on $\la$ may be rewritten as
\begin{equation}\label{1a6a}
\bigg( \eps + \max \Big( \eps\sqrt{\frac2n}, \frac1n \Big) \bigg) V \le 1 \, ;
\end{equation}
it evidently implies $ \eps V < 1 $.

Remark \ref{invar} applies; Theorem \ref{1a5} is scaling
invariant.\footnote{%
 Invariant are $ \eps $, $ \la^2/r $, $ a\la^2 $.}

We start proving Theorem \ref{1a5}. According to Remark \ref{invar} we
restrict ourselves to the case $r=1$.
The following four lemmas are fragments of the proof; they
will not be reused later.
Throughout we assume that $ \eps,a > 0 $, $ f_1(\eps\la) \le a \la^2 $ for $ |\la| \le
1 $, and use $ C \ge \frac{\sqrt2}{\sqrt2-1} $, $ V $ such that
$ \exp \( \frac{\sqrt2}{\sqrt2-1} \eps \) = 1+C\eps $, $ |\la|
= \frac{ 2^{n/2} V }{ n } $, and $ \eps V < 1 $ (that is, $ |\la| <
\frac{2^{n/2}}{\eps n} $).

\begin{lemma}\label{1a7}
Let $ \eps \le \sqrt2-1 $, $ m \in \{0,1,2,\dots\} $, and $ |\la| \le 1
$. Then
\[
f_{2^m} (\eps\la) \le a \la^2 + C \eps (a+1) \la^2 \, .
\]
\end{lemma}

\begin{proof}
Prop.~\ref{1a2} with $ r=1 $ and $ a+1 $ in place of $a$ gives
\[
f_{2^m} ( \eps\la ) \le \( (1+C\eps) (a+1) -
1 \) \la^2 \quad \text{for } |\la| \le 1 \, .
\]
And $ (1+C\eps)(a+1) - 1 = a + C \eps (a+1) $.
\end{proof}

\begin{lemma}\label{1a8}
Let $ \eps \le \sqrt2-1 $, $ m \in \{0,1,2,\dots,n-1\} $, and $ |\la|
\le \frac{ 2^{(n-m)/2} }{ 1 + n 2^{-m/2} \eps } $. Then
\[
f_{2^n} (\eps\la) \le \frac{ a + C \eps ( a + 1 ) }{ 1
- \eps V } \la^2 + \frac{2^{n-m}}n \eps V \, .
\]
\end{lemma}

\begin{proof}
By Lemma \ref{1a7}, $ f_{2^m}(\la) \le \frac{A}{\eps^2} \la^2 $ for
$ |\la| \le \eps $, where $ A = a + C \eps ( a + 1 )
$. Thus, the conditions of Prop.~\ref{1a4} are satisfied for $ r=2^m
$, $ \de = 2^{-m/2} \eps $ and $ a = A/\eps^2 $. Taking also $n-m$ in
place of $n$ we get from Prop.~\ref{1a4}
\[
f_{2^{n-m}2^m} (\la) \le \frac{A}{\eps^2} \frac{\la^2}{ 1
- \frac{(n-m)|\la|}{2^{(n-m)/2}\sqrt{2^m}}} + \frac{ 2^{(n-m)/2} |\la|
}{\sqrt{2^m}}
\]
for $ |\la| \le \frac\de{1+(n-m)\de} 2^{(n-m)/2} \sqrt{2^m} $.
Therefore,
\[
f_{2^n} (\la) \le \frac{A}{\eps^2} \frac{\la^2}{ 1
- \frac{n|\la|}{2^{n/2}} } + 2^{\frac n2 - m} |\la| \quad \text{for }
|\la| \le \frac\de{1+n\de} 2^{n/2} \, .
\]
That is,
\[
f_{2^n} (\eps\la) \le \frac{A}{ 1 - \frac{n\eps|\la|}{2^{n/2}} } \la^2
+ 2^{\frac n2 - m} \eps |\la| = \frac{A}{1-\eps V} \la^2
+ \frac{2^{n-m}}n \eps V
\]
for $ |\la| \le \frac\de{1+n\de} \frac1\eps 2^{n/2} = \frac{
2^{(n-m)/2} }{ 1 + n 2^{-m/2} \eps } $.
\end{proof}

Taking into account that
\[
a + C \eps \frac{ a + 1 + aV }{ 1-\eps V } - \frac{ a + C \eps
( a + 1 ) }{ 1-\eps V } = \frac{ (C-1) \eps aV }{ 1-\eps V
} \ge 0
\]
we get the following.

\begin{corollary}\label{1a9}
Let $ \eps \le \sqrt2-1 $, $ m \in \{0,1,2,\dots,n-1\} $, and $ |\la|
\le \frac{ 2^{(n-m)/2} }{ 1 + n 2^{-m/2} \eps } $. Then
\[
f_{2^n} (\eps\la) \le a \la^2 + C \eps ( a
+ 1 ) \frac{1+V}{1-\eps V} \la^2 - C \eps
\frac{V}{1-\eps V} \la^2 + \frac{2^{n-m}}n \eps V \, .
\]
\end{corollary}

Lemma \ref{1a7} for $ m=n $ gives Theorem \ref{1a5} in the case $
|\la| \le 1 $, that is, $ \frac V n \le 2^{-n/2} $. For greater
$ |\la| $ (and $V$) we'll obtain Theorem \ref{1a5} from
Corollary \ref{1a9}, choosing $m$ as follows. (Recall \eqref{1a6a}.)

\begin{lemma}\label{1a10}
If $ \frac V n > 2^{-n/2} $ and $ \( \eps + \frac1n \) V \le 1
$ (that is, $ 1 < |\la| \le \frac{2^{n/2}}{\eps n + 1} $), then there
exists (evidently unique) $ m \in \{ 0,1,2,\dots,n-1 \} $ such that
\[
1 \le \frac{ (1 - \eps V) n }{ 2^{m/2} V } < \sqrt 2 \, .
\]
\end{lemma}

\begin{proof}
The greatest $ m \in \Z $ such that $ 2^{m/2} \le \frac{
(1 - \eps V) n }{ V } $ satisfies $ m<n $, since $ 2^{n/2}
> \frac n V \ge \frac{ (1 - \eps V) n }{ V } $; it also satisfies
$ m \ge 0 $, since $ \( \eps + \frac1n \) V \le 1 \imply \eps nV +
V \le n \imply 1 \le \frac{ (1 - \eps V) n }{ V } $.
\end{proof}

From now on, $m$ is chosen as above.
Note that $ 1 \le \frac{ (1 - \eps V) n }{ 2^{m/2} V } \imply
(2^{m/2} + \eps n)V \le n \imply  |\la| \le \frac{
2^{(n-m)/2} }{ 1 + n 2^{-m/2} \eps } $, thus,
Corollary \ref{1a9} applies, and so, the next lemma completes the
proof of Theorem \ref{1a5}.

\begin{lemma}
Let $ 1 < |\la| \le \frac{ 2^{n/2} }{ \eps n + \max(\eps\sqrt{2n},1) }
$. Then
\[
\frac{2^{n-m}}n \eps V \le C \eps \frac{V}{1-\eps V} \la^2 \, .
\]
\end{lemma}

\begin{proof}
We rewrite the given restriction $ |\la| \le \frac{ 2^{n/2} }{ \eps n
+ \eps\sqrt{2n} }$ in terms of $V$:
\[
\bigg( 1 + \sqrt{\frac2n} \bigg) \eps V \le 1 \, .
\]
We also eliminate $\la$ from the needed inequality:
\[
C \cdot 2^m V^2 \ge n(1-\eps V) \, .
\]
By \ref{1a10}, $ 2 \cdot 2^m V^2 > n^2 (1 - \eps V)^2 $. Thus, it
is sufficient to prove that $ C n^2 (1 - \eps V)^2 \ge 2n (1-\eps
V) $, that is, $ \eps V \le 1 - \frac2{Cn} $. To this end it is
sufficient to prove that $ \( 1 + \sqrt{\frac2n} \)\( 1
- \frac2{Cn} \) \ge 1 $, that is, $ \sqrt2 n - \frac{2\sqrt2}C \ge
\frac2C \sqrt n $, and we may do it for $n=1$ only: $ \sqrt2
- \frac{2\sqrt2}C \ge \frac2C $, that is, $ \sqrt2 \ge
\frac{2(\sqrt2+1)}C $, since $ C \ge \frac2{2-\sqrt2} $. 
\end{proof}

\subsection{Lower bounds}

In this subsection we investigate an arbitrary family of functions $ f_r
: \R \to [0,\infty] $ for $ r \in (0,\infty) $ such that
\begin{equation}\label{bstart}
f_{2r} (\la) \ge 2p f_r \Big( \frac{\la}{p\sqrt2} \Big) -
\frac{1}{p-1} \cdot \frac{\la^2}{2r}
\end{equation}
whenever $ 0<r<\infty $, $ 1<p<\infty $ and $ \frac{|\la|}{\sqrt{2r}}
\le p-1 $. (The functions \eqref{2a4} satisfy \eqref{bstart} by
Prop.~\ref{2a9}(b).)

If a family $(f_r)_r$ satisfies \eqref{bstart}, then for arbitrary $
s \in (0,\infty) $ the rescaled family $ (g_r)_r $ defined by
\eqref{1.2}, that is, $ g_r (\la) = f_{s^2 r} (s\la) $, satisfies
\eqref{bstart}.

\begin{lemma}\label{1b1}
Let $ a \ge 1 $, $ \eps \ge 0 $, $ r>0 $, and $ \frac{\eps}{\sqrt r}
< \sqrt2 $. If
\[
f_r(\eps\la) \ge (a-1) \la^2 \quad \text{for } |\la| \le 1 \, ,
\]
then
\[
f_{2r} (\eps\la) \ge \bigg( a \Big( 1 - \frac{\eps}{\sqrt{2r}} \Big)
-1 \bigg) \la^2 \quad \text{for } |\la| \le 1 \, .
\]
\end{lemma}

\begin{proof}
We restrict ourselves to the case $ r=1 $ according to Remark
\ref{invar}.\footnote{%
 Invariant are $ \eps/\sqrt r $, $ a $, $ \la $.}
We take
\[
p = \frac1{ 1 - \frac{\eps}{\sqrt2} } \, ,
\]
note that $ \frac{ \eps^2 }{ 2 }
= \( \frac{p-1}p \)^2 \le \frac{(p-1)^2}p $, and apply \eqref{bstart}
to $ \eps\la $ in place of $\la$;
\[
f_2 (\eps\la) \ge 2p(a-1)\frac{\la^2}{2p^2}
- \frac1{p-1} \frac{\eps^2 \la^2}{2} = \Big( \frac{a-1}p
- \frac{ \eps^2 }{ 2(p-1) } \Big) \la^2 \ge \Big( \frac a p -
1 \Big) \la^2 \, . \qedhere
\]
\end{proof}

Iterating the transition $ r \mapsto 2r $ we multiply $a$ by $ \( 1
- \frac{\eps}{\sqrt{2r}} \) \( 1 - \frac{\eps}{\sqrt{4r}} \) \( 1
- \frac{\eps}{\sqrt{8r}} \) \dots $; this product cannot be less than
$ 1 - (\sqrt2+1) \frac{\eps}{\sqrt r} $,
since $ (1-a\eps)(1-b\eps) \ge 1-(a+b)\eps $ for $ a,b \ge 0 $, and
$ \frac{\eps}{\sqrt{2r}} + \frac{\eps}{\sqrt{4r}} + \dots =
(\sqrt2+1) \frac{\eps}{\sqrt r} $. Thus, we get the following.

\begin{proposition}\label{1b2}
Let $ a \ge 1 $, $ \eps \ge 0 $, $ r>0 $, and $ \frac{\eps}{\sqrt r}
< \sqrt2 $. If
\[
f_r(\eps\la) \ge (a-1) \la^2 \quad \text{for } |\la| \le 1 \, ,
\]
then, for every $ n=0,1,2,\dots $,
\[
f_{2^n r} (\eps\la) \ge \bigg( a \Big( 1 - (\sqrt2+1) \frac{\eps}{\sqrt
r} \Big) - 1 \bigg) \la^2 \quad \text{for } |\la| \le 1 \, .
\]
\end{proposition}

\begin{lemma}\label{1b5}
Let $ a,b,c,\de \ge 0 $ and $ r>0 $. If
\[
f_r(\la) \ge \frac{ a\la^2 }{ 1 + \frac{b|\la|}{\sqrt r} } -
\frac{c|\la|}{\sqrt r} \quad \text{for } |\la| \le \de \sqrt r \, ,
\]
then
\[
f_{2r}(\la) \ge \frac{ a\la^2 }{ 1 + \frac{(b+1)|\la|}{\sqrt{2r}} } -
\frac{(2c+1)|\la|}{\sqrt{2r}} \quad \text{for } (1-\de)|\la| \le
\de \sqrt{2r} \, .
\]
(It may be that $ \de\ge1 $, and then $\la$ is not restricted.)
\end{lemma}

\begin{proof} \let\qed\relax
We restrict ourselves to the case $r=1$ according to Remark
\ref{invar}.\footnote{%
 Invariant are $ b $, $ c $, $ \de $, $ \la^2/r $, $ a\la^2 $.}
Assuming $ \la\ne0 $ we take
\[
p =  1 + \frac{|\la|}{\sqrt{2}}  \, ,
\]
note that $ (1-\de)|\la| \le \de\sqrt2 \imply \frac{|\la|}{p\sqrt2}
\le \de $ (also for $\de\ge1$),
and apply \eqref{bstart};
\begin{multline*}
f_2(\la) \ge 2p \bigg( \frac{ \frac{a\la^2}{2p^2} }{ 1 + b
  \big|\frac{\la}{p\sqrt2}\big| } - c \Big|\frac{\la}{p\sqrt2}\Big|
  \bigg) - \frac1{p-1} \frac{\la^2}2 = \\
= \frac{a\la^2}{p + \frac{b|\la|}{\sqrt{2}}} -
 \sqrt 2 c |\la| - \frac{\la^2}{2(p-1)} = \\
= \frac{ a\la^2 }{ 1 + \frac{|\la|}{\sqrt2} + \frac{b|\la|}{\sqrt2} }
- \sqrt 2 c |\la| - \frac{|\la|}{\sqrt2} = 
\frac{ a\la^2 }{ 1 + \frac{(b+1)|\la|}{\sqrt2} } - (2c+1)
  \frac{|\la|}{\sqrt2} \, . \qquad \rlap{$\qedsymbol$}
\end{multline*}
\end{proof}

\begin{proposition}\label{1b6}
Let $ a,\de \ge 0 $, and $ r>0 $. If
\[
f_r(\la) \ge a\la^2 \quad \text{for } |\la| \le \de \sqrt r \, ,
\]
then (for every $ n=0,1,2,\dots $)
\[
f_{2^n r}(\la) \ge \frac{ a\la^2 }{ 1 + \frac{n|\la|}{2^{n/2}\sqrt r} } -
\frac{2^{n/2}|\la|}{\sqrt r} \quad \text{for }
(1-n\de)|\la| \le \de 2^{n/2} \sqrt r \, .
\]
(It may be that $ n\de\ge1 $, and then $\la$ is not restricted.)
\end{proposition}

\begin{proof}
We prove a bit stronger inequality, with the second summand $
-(1-2^{-n}) \frac{2^{n/2}|\la|}{\sqrt r} $ instead of
$ -\frac{2^{n/2}|\la|}{\sqrt r} $, by induction in $n$. Case $n=0$ is
trivial. If the claim holds for $n$, then Lemma \ref{1b5} applies to $
2^n r $, $b=n$, $c=(1-2^{-n}) 2^n = 2^n-1 $, and $ \frac{\de}{1-n\de}
$ (interpreted as $+\infty$ if $ n\de \ge 1 $), giving
\[
f_{2^{n+1} r}(\la) \ge
\frac{ a\la^2 }{ 1 + \frac{(n+1)|\la|}{\sqrt{2^{n+1} r} } } -
\frac{(2^{n+1}-1)|\la|}{\sqrt{2^{n+1} r} }
\]
for $ \(1-\frac{\de}{1-n\de}\) |\la| \le \frac{\de}{1-n\de}
\sqrt{2^{n+1} r} $, that is, $ (1-n\de-\de) |\la| \le \de
\sqrt{2^{n+1} r} $ (and $\la$ is not restricted if $ (n+1)\de \ge 1
$).
\end{proof}

\begin{theorem}\label{1b7}
Let $ \eps \in (0,\sqrt2) $ and $ r,a \in (0,\infty) $. If
\[
f_r(\eps\la) \ge a \la^2 \quad \text{for } |\la| \le \sqrt r \, ,
\]
then, for every $ n = 1,2,\dots $,
\[
f_{2^n r} (\eps\la) \ge a \la^2 - (\sqrt2+1) \eps \Big( a
+ \frac1r \Big) (1+V) \la^2 \qquad \text{for } \;
|\la| \le 2^{n/2} \sqrt r \, ,
\]
where
\[
V = \frac{ n }{ 2^{n/2} } \frac{ |\la| }{ \sqrt r } \, .
\]
\end{theorem}

We start proving Theorem \ref{1b7}. According to Remark \ref{invar} we
restrict ourselves to the case $r=1$.\footnote{%
 Invariant are $ \eps $, $ \la^2/r $, $ a\la^2 $.}
The following two lemmas are fragments of the proof; they will not be
reused later. Throughout we assume that $ \eps,a > 0 $, $ f_1(\eps\la)
\ge a \la^2 $ for $ |\la| \le 1 $, and use $ V $ such that $ |\la| =
\frac{ 2^{n/2} V }{ n } $.

\begin{lemma}\label{1b8}
Let $ \eps < \sqrt2 $, $ m \in \{0,1,2,\dots\} $, and $ |\la| \le 1
$. Then
\[
f_{2^m} (\eps\la) \ge a \la^2 - (\sqrt2+1) \eps (a+1) \la^2 \, .
\]
\end{lemma}

\begin{proof}
Prop.~\ref{1b2} with $ r=1 $ and $ a+1 $ in place of $a$ gives
\[
f_{2^m} ( \eps\la ) \ge \( (1-(\sqrt2+1)\eps) (a+1) -
1 \) \la^2 \quad \text{for } |\la| \le 1 \, .
\]
And $ (1-(\sqrt2+1)\eps)(a+1) - 1 = a - (\sqrt2+1) \eps (a+1) $.
\end{proof}

\begin{lemma}
Let $ \eps < \sqrt2 $, $ m \in \{0,1,2,\dots,n-1\} $, and $ ( 1 -
(n-m) 2^{-m/2} \eps ) |\la| \le 2^{(n-m)/2} $. Then
\[
f_{2^n} (\eps\la) \ge \frac{ a - (\sqrt2+1) \eps ( a + 1 ) }{ 1
+ \eps V } \la^2 - \frac{2^{n-m}}n \eps V \, .
\]
\end{lemma}

\begin{proof}
By Lemma \ref{1b8}, $ f_{2^m}(\la) \ge \frac{A}{\eps^2} \la^2 $ for
$ |\la| \le \eps $, where $ A = a - (\sqrt2+1) \eps ( a + 1 )
$. Thus, the conditions of Prop.~\ref{1b6} are satisfied for $ r=2^m
$, $ \de = 2^{-m/2} \eps $ and $ a = A/\eps^2 $. Taking also $n-m$ in
place of $n$ we get from Prop.~\ref{1b6}
\[
f_{2^{n-m}2^m} (\la) \ge \frac{A}{\eps^2} \frac{\la^2}{ 1
+ \frac{(n-m)|\la|}{2^{(n-m)/2}\sqrt{2^m}}} - \frac{ 2^{(n-m)/2} |\la|
}{\sqrt{2^m}}
\]
for $ (1-(n-m)\de) |\la| \le \de 2^{(n-m)/2} \sqrt{2^m} $.
Therefore,
\[
f_{2^n} (\la) \ge \frac{A}{\eps^2} \frac{\la^2}{ 1
+ \frac{n|\la|}{2^{n/2}} } - 2^{\frac n2 - m} |\la| \quad \text{for }
(1-(n-m)\de) |\la| \le \de 2^{n/2} \, .
\]
That is,
\[
f_{2^n} (\eps\la) \ge \frac{A}{ 1 + \frac{n\eps|\la|}{2^{n/2}} } \la^2
- 2^{\frac n2 - m} \eps |\la| = \frac{A}{1+\eps V} \la^2
- \frac{2^{n-m}}n \eps V
\]
for $ ( 1 - (n-m) 2^{-m/2} \eps ) |\la| \le 2^{(n-m)/2} $.
\end{proof}

Taking into account that
\begin{multline*}
\frac1\eps \Big( a - \frac{ a - (\sqrt2+1) \eps ( a + 1 ) }{ 1+\eps V
 } \Big) = \frac{ (\sqrt2+1) (a+1) + aV }{ 1+\eps V } \le \\
(\sqrt2+1) (a+1) + aV = (\sqrt2+1) (a+1) (1+V) - \( \sqrt2 (a+1) +
 1 \) V
\end{multline*}
and waiving the factor $ 1 - (n-m) 2^{-m/2} \eps $ we get the
following.

\begin{corollary}\label{1b10}
Let $ \eps < \sqrt2 $, $ m \in \{0,1,2,\dots,n-1\} $, and $ |\la| \le
2^{(n-m)/2} $. Then
\[
f_{2^n} (\eps\la) \ge a \la^2 - (\sqrt2+1) \eps (a+1) (1+V) \la^2
+ \eps V \( \sqrt2 (a+1) + 1 \) \la^2 - \frac{2^{n-m}}n \eps V \, .
\]
\end{corollary}

Now we prove Theorem \ref{1b7} as follows. For $ |\la| \le 1 $ we
just apply Lemma \ref{1b8} with $ m=n $. For $ 1 \le |\la| \le 2^{n/2}
$ we choose $ m \in \{ 0,1,2,\dots,n-1 \} $ such that
\[
2^{\frac{n-m-1}2} \le |\la| \le 2^{\frac{n-m}2} \, ,
\]
apply Corollary \ref{1b10} and note that $ \frac{2^{n-m}}n \eps V \le 
\eps V \( \sqrt2 (a+1) + 1 \) \la^2 $, since $ \frac{2^{n-m}}n \le 2
\cdot 2^{n-m-1} \le 2\la^2 \le (\sqrt2+1) \la^2 $.

\subsection{More on the cumulant generating functions}

First, a general fact.

\begin{lemma}\label{2d1}
Let $X$ be a random variable such that $ \Ex \exp |X| < \infty $ and $
\Ex X = 0 $. Then its cumulant generating function
\[
f(\la) = \log \Ex \exp \la X
\]
satisfies
\[
\Big| f(\la) - \frac12 f''(0) \la^2 \Big| \le \frac{41}{6\E^3} \Big(
\frac{|\la|}{1-|\la|} \Big)^3 \( \exp f(-1) + \exp f(1) \) \quad
\text{for } |\la| < 1 \, .
\]
\end{lemma}

\begin{proof}
In terms of $ g(\la) = \Ex \exp \la X $ we have $ f(\la) = \log g(\la)
$ and
\[
f'''(\la) = \frac{g'''(\la)}{g(\la)} - 3 \frac{ g'(\la) g''(\la) }{
  g^2(\la) } + 2 \frac{ g'^3(\la) }{ g^3(\la) } \, .
\]
Applying the inequality $ u^k \E^{-u} \le \( \frac k \E \)^k $ to $ u
= (1-|\la|) |X| $ we get
\begin{multline*}
|g^{(k)}(\la)| = | \Ex X^k \exp \la X | \le \Ex |X|^k \exp |\la| |X| =
\\
= \Ex \( |X|^k \exp(-(1-|\la|) |X|) \exp |X| \) \le \frac1{(1-|\la|)^k}
 \Big( \frac k \E \Big)^k \Ex \exp |X| \, ;
\end{multline*}
also, $ g(\la) \ge 1 $ (since $ \exp(\la X) \ge 1+\la X $);
thus,
\begin{multline*}
|f'''(\la)| \le |g'''(\la)| + 3 |g'(\la)| |g''(\la)| + 2 |g'(\la)|^3
 \le \\
\le \( \Ex \exp |X| \) \bigg( \Big( \frac3{\E(1-|\la|)} \Big)^3 + 3 \Big(
 \frac1{\E(1-|\la|)} \Big) \Big( \frac2{\E(1-|\la|)} \Big)^2 + 2 \Big(
 \frac1{\E(1-|\la|)} \Big)^3 \bigg) \le \\
\le \frac{ 3^3 + 3 \cdot 2^2 + 2 }{
 \E^3 (1-|\la|)^3 } \Ex \exp |X| = \frac{41}{\E^3} \frac1{(1-|\la|)^3}
 \Ex \exp |X| \, .
\end{multline*}
Finally,
\begin{multline*}
\Big| f(\la) - \frac12 f''(0) \la^2 \Big| = \Big| f(\la) - f(0) - f'(0) \la -
 \frac12 f''(0) \la^2 \Big| \le \\
\le \frac1{3!} | f'''(\theta\la) | |\la|^3
\le \frac{41}{6\E^3} \Big( \frac{ |\la| }{ 1-|\la|} \Big)^3 \Ex \exp |X|
\end{multline*}
for some $ \theta \in [0,1] $; and $ \exp |X| \le \exp(-X) + \exp X $.
\end{proof}

We return to the functions $ f_r(\cdot) $ introduced in \eqref{2a4}
for a process $X$ that satisfies Assumption \ref{2a1}.

\begin{proposition}\label{2d2}
There exist $ r_1, \eps \in (0,\infty) $ such that $ f_r(\eps\la) \le
\la^2 $ for all $ r \in [r_1,2r_1] $ and $ \la \in [-1,1] $.
\end{proposition}

\begin{proof}
Def.~\ref{definition1} ensures existence of $ r_1 $ such that $
f_{2r_1}(\cdot) $ is bounded on some $ [-\de,\de] $. Given $ r \in
[r_1,2r_1] $, inequality \ref{2a10}(b) applied to $ r $, $ s=2r_1-r $
and $ p=2 $ gives
\[
2 f_r \bigg( \frac{\la}{2} \sqrt{\frac{r}{2r_1}} \, \bigg) \le
f_{2r_1}(\la) + \frac{\la^2}{2r_1}
\]
for $ |\la| \le \sqrt{2r_1} $. Thus, $ f_r $ is bounded on $
[-\frac{\de}{2\sqrt2}, \frac{\de}{2\sqrt2}] $, uniformly on $ r \in
[r_1,2r_1] $ (assuming $ \de \le \sqrt{2r_1} $; otherwise use $
\min(\de, \sqrt{2r_1}) $). And $ \Ex \exp |\la S_r| \le \exp
f_r(-\la) + \exp f_r(\la) $ is bounded by some $C$ for $ |\la| \le
\frac{\de}{2\sqrt2} $ and $ r \in [r_1,2r_1] $. Using the inequality $
\E^{\eps x}-1 \le \eps (\E^x-1) $ for $ \eps\in[0,1] $ (and all $x$)
we get $ \Ex \exp \eps |\la S_r| \le 1-\eps+\eps \Ex \exp |\la S_r|
\le 1+(C-1)\eps $. We take $\eps$ such that $ 1+(C-1)\eps \le 2 $ and
get $ \Ex \exp \frac{\eps\de}{2\sqrt2}|S_r| \le 2 $ for $ r \in
[r_1,2r_1] $. By Lemma \ref{1aa2}, $ f_r\(\frac{\eps\de}{2\sqrt2}\la\)
\le \la^2 $ for $|\la| \le 1 $.
\end{proof}

\begin{remark}\label{2d2a}
(a) Using Prop.~\ref{1a2} we can serve all $ r \in [r_1,\infty) $ by a
single $\eps$.

(b) On the other hand, $ [r_1,2r_1] $ may be replaced with $ [\theta
r_1,2r_1] $ for arbitrary $ \theta \in (0,1] $ (but a small $\theta$
may require small $\eps$).

(c) Combining (a) and (b) we can serve by a single $\eps$ all $ r \in
[c,\infty) $ for a given $ c>0 $.

(d) In particular, for every $r$ the function $ f_r(\cdot) $ is finite
on some neighborhood of $0$ (but a small $r$ may require small
neighborhood).
\end{remark}

\begin{proposition}\label{2d3}
For every $ r \in (0,\infty) $ there exists (evidently unique) $ \si_r
\in [0,\infty) $ such that for every $ c \in (0,\infty) $ and every $
  \la \in (-c,c) $
\[
\Big| f_r(\la) - \frac12 \si_r^2 \la^2 \Big| \le A \Big(
\frac{|\la|}{c-|\la|} \Big)^3 \( \exp f_r(-c) + \exp f_r(c) \) \, ;
\]
here $A$ is an absolute constant.
\end{proposition}

\begin{proof}
Nothing to prove when the right-hand side is infinite. When it is
finite (which is ensured for small $c$ by \ref{2d2a}(d)) we apply
Lemma \ref{2d1} to the random variable $ cS_r $ and substitute $ \la/c
$ for $\la$. (Of course, $ \si_r^2 = \Ex S_r^2 $, and $ A =
\frac{41}{6\E^3} \approx 0.3402 $ fits.)
\end{proof}

%% file: sect3.tex
\subsection{Quadratic approximation}

In this subsection we investigate an arbitrary family of functions $ f_r
: \R \to [0,\infty] $ for $ r \in (0,\infty) $ that satisfy
\eqref{start}, \eqref{bstart} and Propositions \ref{2d2},
\ref{2d3}. (These assumptions are satisfied by the functions
introduced by \eqref{2a4} for a process $X$ that satisfies Assumption
\ref{2a1}.)

We denote $ \Median(a,b,c) = a+b+c - \min(a,b,c) - \max(a,b,c) $ for $
a,b,c \in \R $.

\begin{theorem}\label{3ath}
Let $ \eps \in (0,\sqrt2-1) $, $ r \in (0,\infty) $, and
\[
f_r(\eps\la\sqrt r) \le \la^2 \quad \text{for } |\la| \le 1 \, .
\]
Then, for every $ n=2,3,\dots $,
\[
\Big| \frac1{\la^2} f_{2^n r}(\eps\la\sqrt r) - \frac12 r \si_{2^n
  r}^2 \eps^2 \Big| \le A \cdot \Median \( |\la|, (2^{-n/2}n\eps|\la|)^{1/3},
2^{-n/2}|\la| \)
\]
for all $\la$ such that $ 0 < |\la| \le 2^{n/2} \min( \frac1{3n\eps},
\frac19 ) $; here $A$ is some absolute constant.
\end{theorem}

Remark \ref{invar} applies: Theorem \ref{3ath} is scaling
invariant.\footnote{%
 Invariant are $ \eps $, $ \la $, and $ r\si^2 $.}

Note that
\begin{multline*}
\Median \( |\la|, (2^{-n/2}n\eps|\la|)^{1/3},
 2^{-n/2}|\la| \) = \\
\begin{gathered}\includegraphics[scale=1.2]{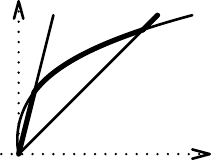}\end{gathered}
\qquad\quad
= \begin{cases}
 |\la| &\text{if } |\la| \le \sqrt{n\eps} \cdot 2^{-n/4},\\
 (2^{-n/2}n\eps|\la|)^{1/3} &\text{if } \sqrt{n\eps} \cdot 2^{-n/4} \le
  |\la| \le \sqrt{n\eps} \cdot 2^{n/2},\\
 2^{-n/2}|\la| &\text{if } \sqrt{n\eps} \cdot 2^{n/2} \le |\la|.
\end{cases}
\end{multline*}

\begin{lemma}\label{3a1}
Let $ \eps \le \sqrt2-1 $. If $ f_1(\eps\la) \le \la^2 $ for $
|\la|\le1 $, then
\[
\Big| f_{2^n}(\eps\la) - \frac12 \si_{2^n}^2 \eps^2 \la^2 \Big| \le A
\Big( \frac{|\la|}{1-|\la|} \Big)^3 \quad \text{for } |\la| < 1 \, ;
\]
here $A$ is an absolute constant.
\end{lemma}

\begin{proof}
By Prop.~\ref{1a2} for $ a=2 $, $ r=1 $ and $ \la=\pm1 $,
\[
f_{2^n}(\pm\eps) \le 2 \exp \Big( \frac{\sqrt2}{\sqrt2-1} \eps \Big) -
1 \le 2\E^{\sqrt2} - 1 \, .
\]
By Prop.~\ref{2d3},
\[
\Big| f_{2^n}(\la) - \frac12 \si_{2^n}^2 \la^2 \Big| \le
A_\text{\ref*{2d3}}
\Big( \frac{|\la|}{\eps-|\la|} \Big)^3 \cdot 2 \exp ( 2\E^{\sqrt2} - 1
) \quad \text{for } |\la| < \eps \, ,
\]
that is,
\[
\Big| f_{2^n}(\eps\la) - \frac12 \si_{2^n}^2 \eps^2 \la^2 \Big| \le A
\Big( \frac{\eps|\la|}{\eps-\eps|\la|} \Big)^3 \quad \text{for } |\la|
< 1 \, .
\]
where $ A = 2 \exp ( 2\E^{\sqrt2} - 1 ) A_\text{\ref*{2d3}} $.
\end{proof}

\begin{lemma}
$ | \si_{2r} - \si_r | \le \frac1{\sqrt r} $ for all $ r \in
  (0,\infty) $.
\end{lemma}

\begin{proof}
Taking into account that $ f_r(\la) = \frac12 \si_r^2 \la^2 + o(\la^2)
$, we get from \eqref{start}
\[
\frac12 \si_{2r}^2 \le p \cdot \frac12 \si_r^2 + \frac{p}{p-1} \cdot
\frac1{2r} \quad \text{for all } p \in (1,\infty) \, .
\]
Taking $ p = 1 + \frac1{\si_r \sqrt r} $ (the minimizer, in fact) we
get $ \frac{p}{p-1} = 1 + \si_r \sqrt r $, $ \si_{2r}^2 \le \(
\si_r + \frac1{\sqrt r} \)^2 $, thus, $ \si_{2r} \le \si_r +
\frac1{\sqrt r} $. It remains to prove that $ \si_{2r} \ge \si_r -
\frac1{\sqrt r} $.

By \eqref{bstart},
\[
\frac12 \si_{2r}^2 \ge \frac1p \cdot \frac12 \si_r^2 - \frac{1}{p-1}
\cdot \frac1{2r} \, ;
\]
assuming $ \si_r > \frac1{\sqrt r} $ (otherwise we have nothing to
prove), we take $ p = \frac{\si_r \sqrt r}{\si_r \sqrt r - 1} $ (the
minimizer), get $ \frac1p = 1 - \frac1{\si_r \sqrt r} $, $ \frac{1}{p-1} =
\si_r \sqrt r - 1 $, and $ \si_{2r}^2 \ge \( \si_r - \frac1{\sqrt r}
\)^2 $.
\end{proof}

\begin{corollary}\label{3a3}
$ | \si_{2^n r} - \si_{2^m r} | \le \frac{\sqrt2}{\sqrt2-1}
\frac1{\sqrt{2^mr}} $ whenever $ m \le n $ and $ r \in (0,\infty) $.
\end{corollary}

\begin{proposition}\label{3a4}
Let $ \eps \in (0,\sqrt2-1) $, and
\[
f_1(\eps\la) \le \la^2 \quad \text{for } |\la| \le 1 \, .
\]
Then
\begin{multline*}
\bigg| \frac1{\la^2} f_{2^n}(\eps\la) - \frac12 \si_{2^n}^2 \eps^2
 \bigg| \le \frac12 \eps^2 \bigg( \Big( \si_{2^n} +
 \frac{\sqrt2}{\sqrt2-1} 2^{-m/2} \Big)^2 - \si_{2^n}^2 \bigg) + \\
+ A \de + A \eps \( n 2^{-n/2} |\la| + 2^{-m/2} \de \) \de^{-2}
\end{multline*}
for all $ m \in \{0,1,\dots,n-1\} $, $ \de \in (0,\frac12] $ and $ \la
$ such that
\[
0 < |\la| \le \frac{ 2^{n/2} \de }{ 2n\eps\de + \max( \sqrt{2n} \de\eps,
  2^{m/2} ) } \, ;
\]
here $A$ is some absolute constant.
\end{proposition}

\begin{proof}
Using Lemma \ref{3a1} we have for $ |\la| \le \de $
\[
\Big| f_{2^m}(\eps\la) - \frac12 \si_{2^m}^2 \eps^2 \la^2 \Big| \le
A_\text{\ref*{3a1}}
\Big( \frac{|\la|}{1-|\la|} \Big)^3 \le A_\text{\ref*{3a1}}
\frac{\de}{(1-\de)^3} \la^2 \le 8 A_\text{\ref*{3a1}} \de \la^2 \, ,
\]
thus,
\begin{equation}\label{3a5}
f_{2^m}(\eps\la) \le \Big( \frac12 \si_{2^m}^2 \eps^2 + 8
A_\text{\ref*{3a1}} \de \Big) \la^2 \quad \text{for } |\la| \le \de \,
.
\end{equation}
Now we need Theorem \ref{1a5} rescaled as follows ($ a $ will be
chosen later):
\[
\eps_\text{\ref*{1a5}} = 2^{-m/2} \de \eps \, , \quad\!\!
r_\text{\ref*{1a5}} = 2^m \, , \quad\!\! a_\text{\ref*{1a5}} = 2^{-m} a \, ,
\quad\!\! n_\text{\ref*{1a5}} = n-m \, , \quad\!\! \la_\text{\ref*{1a5}} =
2^{m/2} \de^{-1} \la \, .
\]
We note that
\[
(\eps\la)_\text{\ref*{1a5}} = \eps\la \, , \quad
(a\la^2)_\text{\ref*{1a5}} = \de^{-2}a\la^2 \, , \quad (\la/\sqrt
r)_\text{\ref*{1a5}} = \de^{-1}\la \, , \quad (2^n r)_\text{\ref*{1a5}}
= 2^n \, . \quad
\]
The assumptions of \ref{1a5} become: $ 2^{-m/2} \de \eps < \sqrt2-1 $
(holds evidently), and
\[
f_{2^m}(\eps\la) \le \de^{-2} a \la^2 \quad \text{for } |\la| \le \de
\, ;
\]
the latter holds by \eqref{3a5} provided that
\begin{equation}\label{3a6}
a = \Big( \frac12 \si_{2^m}^2 \eps^2 + 8
A_\text{\ref*{3a1}} \de \Big) \de^2 \, .
\end{equation}
The conclusion of \ref{1a5} becomes
\begin{equation}\label{3a7}
f_{2^n}(\eps\la) \le \Big( a + C_\text{\ref*{1a5}} \cdot 2^{-m/2}
\de\eps (a+1) \frac{1+V}{1-2^{-m/2}\de\eps V} \Big) \de^{-2} \la^2 
\end{equation}
for $\displaystyle |\la| \le \frac{ 2^{n/2} \de }{ (n-m)\eps\de + \max(
\sqrt{2(n-m)} \de\eps, 2^{m/2} ) } $ (which holds evidently),
where $a$ is given by \eqref{3a6}, $ V = \frac{n-m}{2^{(n-m)/2}}
\de^{-1} |\la| $, and $ C_\text{\ref*{1a5}} \le
\frac{2(\E^{1/\sqrt2}-1)}{\sqrt2-1} $ (since $ \eps_\text{\ref*{1a5}}
\le \frac{\sqrt2-1}2 $).

We have to prove two bounds, upper and lower, on $ f_{2^n}(\eps\la)
$. For the upper bound, it is sufficient to prove that
\begin{multline*}
\frac12 \si_{2^m}^2 \eps^2 + 8 A_\text{\ref*{3a1}} \de + \Big(
 C_\text{\ref*{1a5}} \cdot 2^{-m/2} \de\eps (a+1)
 \frac{1+V}{1-2^{-m/2}\de\eps V} \Big) \de^{-2} \le \\
\le \frac12 \eps^2 \Big( \si_{2^n} + \frac{\sqrt2}{\sqrt2-1} 2^{-m/2}
 \Big)^2 + A \de + A \eps \( n 2^{-n/2} |\la| + 2^{-m/2} \de \)
 \de^{-2} \, .
\end{multline*}
By Corollary \ref{3a3}, $ \si_{2^m} \le \si_{2^n} +
\frac{\sqrt2}{\sqrt2-1} 2^{-m/2} $; the needed upper bound inequality
is reduced to
\begin{multline*}
8 A_\text{\ref*{3a1}} \de + \Big(
 \frac{2(\E^{1/\sqrt2}-1)}{\sqrt2-1} \cdot 2^{-m/2} \de\eps (a+1)
\frac{1+V}{1-2^{-m/2}\de\eps V} \Big) \de^{-2} \le \\
\le A \de + A \eps \( n 2^{-n/2} |\la| + 2^{-m/2} \de \)
 \de^{-2} \, ,
\end{multline*}
and further to\footnote{%
 From now on, $A$ denotes different absolute constants in different
 inequalities.}
\begin{equation}\label{3a75}
2^{-m/2} \de (a+1) \frac{1+V}{1-2^{-m/2}\de\eps V} \le A \( n
2^{-n/2} |\la| + 2^{-m/2} \de \) \, .
\end{equation}
We note that $ \si_1 \le \sqrt2/\eps $ (since $ f_1(\eps\la) \le \la^2
$); \ref{3a3} gives $ \eps\si_{2^k} \le \eps \( \frac{\sqrt2}\eps +
\frac{\sqrt2}{\sqrt2-1} \) \le 2\sqrt2 $ for all $k$. By \eqref{3a6},
$ a \le 1+A_\text{\ref*{3a1}} $ (since $ \de \le \frac12 $), which
reduces \ref{3a75} to
\[
2^{-m/2} \de \frac{1+V}{1-2^{-m/2}\de\eps V} \le A \( n
2^{-n/2} |\la| + 2^{-m/2} \de \) \, .
\]
Further, $ |\la| \le \frac{2^{n/2}}{2n\eps} $, thus
\[
2^{-m/2} \de\eps V = 2^{-m/2} \de\eps \frac{n-m}{2^{(n-m)/2}} \de^{-1}
|\la| = \eps(n-m)2^{-n/2}|\la| \le \frac12 \, ,
\]
which reduces \ref{3a75} to
\[
2^{-m/2} \de (1+V) \le A \( n 2^{-n/2} |\la| + 2^{-m/2} \de \) \, ,
\]
and further, to $ 2^{-m/2} \de V \le A n 2^{-n/2} |\la| $, which holds
(for $A=1$) by the definition of $V$.

For the lower bound the proof is similar. First,
\begin{equation}\label{3a8}
f_{2^m}(\eps\la) \ge \Big( \frac12 \si_{2^m}^2 \eps^2 - 8
A_\text{\ref*{3a1}} \de \Big) \la^2 \quad \text{for } |\la| \le \de
\end{equation}
similarly to \eqref{3a5}. Second, the rescaling that was applied to
Th.~\ref{1a5} applies now to Theorem \ref{1b7}. The assumptions of
\ref{1b7} become: $ 2^{-m/2} \de \eps < \sqrt2 $ (holds evidently),
and
\[
f_{2^m}(\eps\la) \ge \de^{-2} a \la^2 \quad \text{for } |\la| \le \de
\, ;
\]
the latter holds by \eqref{3a8} provided that
\begin{equation}\label{3a9}
a = \Big( \frac12 \si_{2^m}^2 \eps^2 - 8
A_\text{\ref*{3a1}} \de \Big) \de^2 \, .
\end{equation}
The conclusion of \ref{1b7} becomes
\begin{equation}\label{3a10}
f_{2^n}(\eps\la) \ge \( a - (\sqrt2+1) 2^{-m/2}
\de\eps (a+1) (1+V) \) \de^{-2} \la^2 
\end{equation}
for $ 2^{m/2} |\la| \le 2^{n/2} \de $ (which
holds, since $ |\la| \le 2^{(n-m)/2} \de $), where $a$ is given by
\eqref{3a9}, and $ V = \frac{n-m}{2^{(n-m)/2}} \de^{-1} |\la| $ as
before.

We replace $ \Big( \si_{2^n} + \frac{\sqrt2}{\sqrt2-1} 2^{-m/2}
\Big)^2 - \si_{2^n}^2 $ with $ \si_{2^n}^2 - \Big( \si_{2^n} -
\frac{\sqrt2}{\sqrt2-1} 2^{-m/2} \Big)_+^2 $ (the latter being
smaller). It is sufficient to prove that
\begin{multline*}
\frac12 \si_{2^m}^2 \eps^2 - 8 A_\text{\ref*{3a1}} \de - 
 (\sqrt2+1) 2^{-m/2} \de\eps (a+1)
 (1+V) \de^{-2} \ge \\
\ge \frac12 \eps^2 \Big( \si_{2^n} - \frac{\sqrt2}{\sqrt2-1} 2^{-m/2}
 \Big)_+^2 - A \de - A \eps \( n 2^{-n/2} |\la| + 2^{-m/2} \de \)
 \de^{-2} \, .
\end{multline*}
By Corollary \ref{3a3}, $ \si_{2^m} \ge \si_{2^n} -
\frac{\sqrt2}{\sqrt2-1} 2^{-m/2} $; the needed lower bound inequality
is reduced to
\begin{multline*}
8 A_\text{\ref*{3a1}} \de +
 (\sqrt2+1) 2^{-m/2} \de\eps (a+1)
 (1+V) \de^{-2} \le \\
\le A \de + A \eps \( n 2^{-n/2} |\la| + 2^{-m/2} \de \)
 \de^{-2} \, ,
\end{multline*}
and further to
\[
2^{-m/2} \de (a+1) (1+V) \le A \( n 2^{-n/2}
|\la| + 2^{-m/2} \de \) \, ;
\]
the latter holds by \eqref{3a75}.
\end{proof}

\smallskip

We start proving Theorem \ref{3ath}. According to Remark \ref{invar}
we restrict ourselves to the case $ r=1 $.

\smallskip

\textsc{The first case:} $ |\la| \le \sqrt{n\eps} \cdot 2^{-n/4} $.

We note that $ \sqrt{n\eps} \cdot 2^{-n/4} \le \( (\sqrt2-1) \max_n n
2^{-n/2} \)^{1/2} = \( (\sqrt2-1) \cdot 3 \cdot 2^{-3/2} \)^{1/2} <
2/3 $ and apply Lemma \ref{3a1}: $ \big| \frac1{\la^2} f_{2^n}(\eps\la) -
\frac12 \si_{2^n}^2 \eps^2 \big| \le A_\text{\ref*{3a1}}
\frac{|\la|}{(1-|\la|)^3} \le 27 A_\text{\ref*{3a1}} |\la| $.

\smallskip

\textsc{The second case:} $ \sqrt{n\eps} \cdot 2^{-n/4} \le |\la| \le
2^{n/2} \min \( \sqrt{n\eps}, \frac1{3n\eps}, \frac19 \) $.

Before applying Prop.~\ref{3a4} we choose $ m \in \{0,1,\dots,n-1\} $
and $ \de \in (0,\frac12] $ appropriately; namely, we want them to
satisfy
\begin{gather}
\frac{\de}3 \le (2^{-n/2}n\eps|\la|)^{1/3} \le 3\de \,
 , \label{3a13} \\
\frac1{\sqrt2} \de \le 3 \cdot 2^{-(n-m)/2} |\la| \le \de \,
 . \label{3a14}
\end{gather}

\begin{lemma}\label{3a15}
Let $ 1 \le a < \infty $, $ 0 < x < \infty $, $ 0 < \la \le \min
( \sqrt x, 1/x, 1/a^2 ) $. Then there exists $ \de \in (0,\infty) $
such that $ a \de \le 1 $, $ a \la \le \de $, and
\[
\frac{\de}{a} \le (x\la)^{1/3} \le a\de \, .
\]
\end{lemma}

\begin{proof}
Existence of $ \de $ such that $ \de \le a^{-1} $, $ \de \ge a\la $,
$ \de \ge a^{-1} (x\la)^{1/3} $, $ \de \le a (x\la)^{1/3} $ is
equivalent to the inequality
\[
\max ( a\la, a^{-1} (x\la)^{1/3} ) \le \min ( a^{-1}, a (x\la)^{1/3}
) \, ,
\]
thus, to the three inequalities
\begin{gather*}
a\la \le a^{-1} \, , \quad \text{that is,} \quad \la \le a^{-2} \,
 ; \\
a\la \le a (x\la)^{1/3} \, , \quad \text{that is,} \quad \la \le \sqrt
 x \, ; \\
a^{-1} (x\la)^{1/3} \le a^{-1} \, , \quad \text{that
 is,} \quad \la \le x^{-1} \, . \qedhere
\end{gather*}
\end{proof}

Lemma \ref{3a15}, applied to $ a = 3 $,
$x=n\eps$ and $ 2^{-n/2} |\la| $ in place of $ \la $, gives $ \de $
such that $ 3\de \le 1 $ (and therefore $ \de < \frac12 $, as
required), $ 3 \cdot 2^{-n/2} |\la| \le \de $, and \eqref{3a13} holds.

By \eqref{3a13}, $ \( \frac{\de}3 \)^3 \le 2^{-n/2} n\eps |\la| $;
on the other hand, $ 2^{-n/2} n\eps \le |\la|^2 $; therefore $ \(
\frac{\de}3 \)^3 \le |\la|^3 $, that is, $ \de \le 3|\la| $.

Having $ \frac{\de}{3|\la|} \in [2^{-n/2},1] = \cup_{m=0}^{n-1}
[2^{-(n-m)/2}, \sqrt2 \cdot 2^{-(n-m)/2}] $, we take $m$ such that $ 
\frac{\de}{3|\la|} \in [2^{-(n-m)/2}, \sqrt2 \cdot 2^{-(n-m)/2}] $,
which ensures \eqref{3a14}.

In order to apply Prop.~\ref{3a4} we have to check that
\begin{equation}\label{3a16}
|\la| \le \frac{ 2^{n/2} \de }{ 2n\eps\de + \max( \sqrt{2n} \de\eps,
2^{m/2} ) } \, .
\end{equation}
We know that $ |\la| \le \frac{2^{n/2}}{3n\eps} $; also, $ |\la| \le
\frac13 \cdot 2^{(n-m)/2} \de $ by \eqref{3a14}; thus, \eqref{3a16} is
reduced to
\[
\min \Big( \frac1{3n\eps}, \frac13 \cdot 2^{-m/2} \de \Big) \cdot \max \(
2n\eps\de + \sqrt{2n} \eps \de, 2n\eps\de + 2^{m/2} \) \le \de \, ,
\]
that is,
\[
\min \Big( \frac1{3n\eps}, \frac13 \cdot 2^{-m/2} \de \Big) \cdot \max \Big(
2n \( 1+\tfrac1{\sqrt{2n}} \) \eps, 2n\eps + 2^{m/2} \de^{-1} \Big)
\le 1 \, .
\]
The left-hand side does not exceed\footnote{%
 Since $ n \ge 2 $, and $ \min(x,y) \cdot \max(u,v+w) \le
 \max(xu,xv+yw) $ for $ u,v,w \ge 0 $.}
\[
\max \Big( \frac1{3n\eps} \( 1+\tfrac12 \) \cdot 2n\eps,
\frac1{3n\eps} \cdot 2n\eps + \frac13 \cdot 2^{-m/2} \de \cdot 2^{m/2}
\de^{-1} \Big) = \max (1,1) = 1 \, .
\]
So, \eqref{3a16} holds; Prop.~\ref{3a4} applies, and gives the upper
bound
\[
\frac12 \eps^2 \bigg( \Big( \si_{2^n} + \frac{\sqrt2}{\sqrt2-1}
2^{-m/2} \Big)^2 - \si_{2^n}^2 \bigg) + A_\text{\ref*{3a4}} \de +
A_\text{\ref*{3a4}} \eps \( n 2^{-n/2} |\la| + 2^{-m/2} \de \) \de^{-2}
\, ;
\]
we want to majorize this by $ \const \cdot (2^{-n/2}n\eps|\la|)^{1/3} $ or,
equivalently, by $ \const \cdot \de $ (see \eqref{3a13}).

Below, $ \cO(x) $ means something majorized by $ \const \cdot \, x $ with
some \emph{absolute} constant. We have
\begin{align}
& \de = \cO(1) && \text{\small since } \de \le \tfrac12 \, ; \label{3a17} \\
& \eps \si_{2^n} = \cO(1) && \text{\small since } \eps \si_{2^n} \le 2\sqrt2 \,
 , \text{ \small as noted after \eqref{3a75}} \, ; \label{3a18} \\
& 2^{-n/2} n \eps |\la| = \cO(\de^3) && \text{\small by \eqref{3a13}: }
 2^{-n/2} n \eps |\la| \le (3\de)^3 \, ; \label{3a19} \\
& 2^{-m/2} = \cO(2^{-n/2}|\la|\de^{-1}) && \text{\small by
 \eqref{3a14}: } 2^{-m/2} \le \sqrt2 \cdot 3 \cdot 2^{-n/2}|\la|\de^{-1} \,
 ; \label{3a20} \\
& 2^{-m/2}\eps = \cO(\de^2) && \text{\small by \eqref{3a20} and
  \eqref{3a19}} \, . \label{3a21}
\end{align}
Thus, by \eqref{3a21} and \eqref{3a18},
\[
\eps^2 \bigg( \Big( \si_{2^n} + \frac{\sqrt2}{\sqrt2-1} 2^{-m/2}
\Big)^2 - \si_{2^n}^2 \bigg) =
\underbrace{ \cO( 2^{-m/2}\eps \cdot \eps \si_{2^n} ) }_{=\cO(\de^2)} +
\underbrace{ \cO( 2^{-m} \eps^2 ) }_{=\cO(\de^4)} = \cO(\de^2) =
\cO(\de) \, ;
\]
and finally, by \eqref{3a19} and \eqref{3a21},
\[
\eps \( n 2^{-n/2} |\la| + 2^{-m/2} \de \) \de^{-2} = \( \cO(\de^3) +
\cO(\de^2) \de \) \de^{-2} = \cO(\de) \, .
\]

\smallskip

\textsc{The third case:} $ 2^{n/2} \sqrt{n\eps} \le |\la| \le
2^{n/2} \min( \frac1{3n\eps}, \frac19 ) $.

We want to apply Prop.~\ref{3a4} for $ m=0 $ and $ \de = 3 \cdot
2^{-n/2}|\la| $ (as required, $ \de < \frac12 $). To this end we check
that
\[
|\la| \le \frac{ 3|\la| }{ 6n\eps \cdot 2^{-n/2}|\la| + \max( 3 \sqrt{2n}
  2^{-n/2}|\la| \eps, 1 ) } \, ,
\]
that is,
\[
\max \( (2n+\sqrt{2n}) \eps 2^{-n/2}|\la|, \, 2n\eps 2^{-n/2}|\la| +
\tfrac13 \) \le 1 \, .
\]
The left-hand side does not exceed
\[
\max \Big( 2n \( 1+\tfrac12 \) \eps \cdot
\frac1{3n\eps}, 2n\eps \cdot \frac1{3n\eps} + \frac13
\Big) = \max (1,1) = 1 \, .
\]
So, Prop.~\ref{3a4} applies, and gives the upper bound
\begin{multline*}
\frac12 \eps^2 \bigg( \Big( \si_{2^n} + \frac{\sqrt2}{\sqrt2-1}
 \Big)^2 - \si_{2^n}^2 \bigg) + \\
+ 3 A_\text{\ref*{3a4}} 2^{-n/2}|\la| +
 A_\text{\ref*{3a4}} \eps \( n 2^{-n/2} |\la| + 3 \cdot 2^{-n/2}|\la|
 \) (3 \cdot 2^{-n/2}|\la|)^{-2} = \\
= \cO(\eps^2 \si_{2^n}) + \cO(\eps^2) + \cO(2^{-n/2}|\la|) +
 \cO\Big(\frac{\eps n2^{n/2}}{|\la|}\Big) +
 \cO\Big(\frac{2^{n/2}\eps}{|\la|}\Big) \, ;
\end{multline*}
we want to majorize this by $ 2^{-n/2}|\la| $.

We have
\begin{align}
& 2^{-n/2}|\la| = \cO(1) && \text{\small since } |\la| \le 2^{n/2} \cdot
 \tfrac19 \, ; \label{3a22} \\
& \eps \le n\eps = \cO\( (2^{-n/2}|\la|)^2 \) && \text{\small since }
 2^{n/2} \sqrt{n\eps} \le |\la| \, ; \label{3a23} \\
& \eps = \cO(2^{-n/2}|\la|) && \text{\small by \eqref{3a23} and
 \eqref{3a22}} \, ; \label{3a24} \\
& \eps^2 \si_{2^n} = \cO(2^{-n/2}|\la|) && \text{\small by
   \eqref{3a24} and \eqref{3a18}} ; \notag \\
& \eps^2 = \cO(2^{-n/2}|\la|) && \text{\small by \eqref{3a24} and
 \eqref{3a22}} ; \notag
\end{align}
and finally, $ \frac{\eps n2^{n/2}}{|\la|} \le \frac{
  (2^{-n/2}|\la|)^2 \cdot 2^{n/2} }{ |\la| } = 2^{-n/2}|\la| $.

Theorem \ref{3ath} is proved.

\begin{corollary}\label{3cor}
Under the assumptions of Theorem \ref{3ath},
\[
\Big| \frac1{\la^2} f_{2^n r}(\eps\la\sqrt r) - \frac12 r \si_{2^n
  r}^2 \eps^2 \Big| \le A (2^{-n/2}n\eps|\la|)^{1/3}
\]
for all $\la$ such that $ 0 < |\la| \le 2^{n/2} \min( \frac1{3n\eps},
\frac19, \sqrt{n\eps} ) $.
\end{corollary}

\subsection{Main result: proof}

We return to the numbers $ \si_r $ introduced by Prop.~\ref{2d3}. For
every $ r \in (0,\infty) $ the limit
\[
\si_{2^\infty r} = \lim_{n\to\infty} \si_{2^n r}
\]
exists by \ref{3a3}.

\begin{lemma}
$ \si_{2^\infty r} $ does not depend on $r$.
\end{lemma}

\begin{proof}
We'll prove that the function $ r \mapsto r\si^2_{2^\infty r} $ is
linear. It is sufficient to prove that it is additive,
\begin{equation}\label{*}
(r+s) \si^2_{2^\infty (r+s)} = r \si^2_{2^\infty r} + s
\si^2_{2^\infty s} \, ,
\end{equation}
and measurable.

For every $\la$ the function $ r \mapsto f_r(\la) $ is measurable due
to \eqref{01.2}, which implies measurability of the functions $ r
\mapsto \si_r^2 = \lim_{\la\to0} \frac2{\la^2} f_r(\la) $ and $ r
\mapsto \si_{2^\infty r} $.

Multiplying by $ \frac2{\la^2} $ the inequality \ref{2a10}(a) and
taking the limit as $ \la \to 0 $ we get
\[
\si^2_{r+s} \le \frac1p \cdot p^2 \frac{r}{r+s} \si_r^2 + \frac1p
\cdot p^2 \frac{s}{r+s} \si_s^2 + \frac{p}{p-1} \frac{2}{r+s} \, ;
\]
applying it to $ 2^n r $, $ 2^n s $ and taking the limit as $ n \to
\infty $ we get
\[
\si^2_{2^\infty(r+s)} \le p \frac{r}{r+s} \si_{2^\infty r}^2 + p
\frac{s}{r+s} \si_{2^\infty s}^2
\]
for all $ p>1 $ and therefore for $ p=1 $. Similarly, the inequality
\[
(r+s) \si^2_{2^\infty(r+s)} \ge r \si_{2^\infty r}^2 + s \si_{2^\infty
s}^2
\]
follows from \ref{2a10}(b), and we get \eqref{*}.
\end{proof}

Now we have $ \si \in [0,\infty) $ such that $ \si_{2^n r} \to \si $
(as $ n \to \infty $) for all $ r \in (0,\infty) $; applying \ref{3a3}
to $ m=0 $ and $ n \to \infty $ we get
\begin{equation}\label{3a27}
| \si_r - \si | \le \frac{\sqrt2}{\sqrt2-1} \frac1{\sqrt r} \, ;
\qquad\quad
\si_r \to \si \quad \text{\small as } r \to \infty \, .
\end{equation}

\begin{proof}[{\bfseries Proof} {\sffamily of Theorem \ref{theorem1}}]

Assumption \ref{2a1} applies without loss of generality. Remark
\ref{2d2a}(c) gives $\eps$ such that
\[
f_r(\eps\la) \le \la^2 \qquad \text{for all } \la \in [-1,1] \text{
  and } r \in [\tfrac12, 1] \, ,
\]
which ensures the condition of Th.~\ref{3ath}: $ f_r (\eps\la\sqrt
r) \le \la^2 $ for these $ \la $ and $ r $ (if $ \eps < \sqrt2 - 1 $;
otherwise take a smaller $\eps$). Corollary \ref{3cor} applied to $
\eps $ and $ 2^{-n} r $ gives, whenever $ 2^{-n} r \in [\frac12,1] $,
\[
\Big| \frac1{\la^2} f_r(\eps\la2^{-n/2}\sqrt r) - \frac12 \cdot 2^{-n}
r \si_r^2 \eps^2 \Big| \le A (2^{-n/2}n\eps|\la|)^{1/3}
\]
for $ 0 < |\la| \le 2^{n/2} \min( \frac1{3n\eps}, \frac19,
\sqrt{n\eps} ) $. We replace $\la$ with $ 2^{n/2} \la / \eps $:
\[
\Big| \frac{\eps^2}{2^n \la^2} f_r(\la\sqrt r) - \frac12 \cdot 2^{-n}
r \si_r^2 \eps^2 \Big| \le A (n|\la|)^{1/3}
\]
for $ 0 < |\la| \le \min( \frac1{3n}, \frac{\eps}{9}, \eps\sqrt{n\eps}
) $. Thus,
\[
\Big| \frac1{r \la^2} f_r(\la\sqrt r) - \frac12 \si_r^2 \Big| \le A
\cdot \frac{2^n}{\eps^2 r} (n|\la|)^{1/3} \le \frac{2A}{\eps^2}
(n|\la|)^{1/3} \le \frac{2A}{\eps^2} \Big( |\la| \frac{\log 2r}{\log
  2} \Big)^{1/3}
\]
is small whenever $r$ is large and $ |\la| \log r $ is small. Also, $
\si_r^2 $ is close to $ \si^2 $ by \eqref{3a27}.
\end{proof}

\begin{proof}[{\bfseries Proof} {\sffamily of Corollary \ref{corollary3}}]
Let $ r_n \to \infty $, $ c_n \to \infty $, $ (c_n \log r_n )^2 / r_n \to 0 $;
we have to prove that
\[
\frac1{c_n^2} \log \PR{ \int_0^{r_n} X_t \, \D t \ge c_n \si \sqrt{r_n}
} \to -\frac12 \quad \text{as } n \to \infty \, .
\]
Theorem \ref{theorem1} applied to $ r_n $ and $ \la_n = \la c_n /
\sqrt{r_n} $ gives
\[
\frac1{c_n^2} \log \Ex \exp \frac{ \la c_n }{ \sqrt{r_n} }
\int_0^{r_n} X_t \, \D t \to \frac{\si^2}2 \la^2 \quad \text{as } n
\to \infty
\]
for all $ \la \in \R $. By the G\"artner-Ellis theorem \cite{Ell} (with the
scale $ c_n $ and speed $ c_n^2 $), random variables $ \frac1{ c_n
\sqrt{r_n} } \int_0^{r_n} X_t \, \D t $ satisfy MDP with the rate function $
x \mapsto \frac{ x^2 }{ 2\si^2 } $.
\end{proof}

\begin{proof}[{\bfseries Proof} {\sffamily of Corollary \ref{corollary4}}]
For every $ \la \ne 0 $ Theorem \ref{theorem1} applied to $ r $ and $
\la/\sqrt r $ gives
\[
\frac1{ r (\la/\sqrt r)^2 } \log \Ex \exp \frac{\la}{\sqrt r} \int_0^r
X_t \, \D t \to \frac{\si^2}2 \quad \text{as } r \to \infty \, ,
\]
that is,
\[
\Ex \exp \la \cdot \frac1{\sqrt r} \int_0^r X_t \, \D t \to \exp \Big(
\frac12 \si^2 \la^2 \Big) \quad \text{as } t \to \infty \, .
\]
The weak convergence of distributions follows, see for example \cite[Sect.~30,
p.~390]{Bi}.
\end{proof}

%% file: main.bbl
\begin{thebibliography}{8.}

{\raggedright
\bibitem{Bi} P. Billinglsley (1995):
\emph{Probability and measure} (third edition),
Wiley.

\bibitem{Ell} R.S. Ellis (2006):
\emph{The theory of large deviations and applications to statistical
 mechanics,}
\href{http://www.math.umass.edu/~rsellis/pdf-files/Dresden-lectures.pdf}%
{http://www.math.umass.edu/$\sim$rsellis/pdf-files/Dresden-lectures.pdf}

}
\end{thebibliography}
